\documentclass[10pt]{amsart}
\input{epsf}
\usepackage{amssymb,latexsym}
\usepackage{enumitem}
\topskip  \numberwithin{equation}{section}
\newtheorem{theorem}{Theorem}[section]
\newtheorem{proposition}[theorem]{Proposition}
\newtheorem{corollary}[theorem]{Corollary}

\newtheorem{lemma}[theorem]{Lemma}

\begin{document}

\pagenumbering{arabic}
\pagestyle{headings}
\def\sof{\hfill\rule{2mm}{2mm}}
\def\llim{\lim_{n\rightarrow\infty}}

\title[Chebyshev polynomials and statistics]{Chebyshev polynomials and statistics on a new collection of words in the Catalan family}

\maketitle

\begin{center}

\author{Toufik Mansour\\
\small Department of Mathematics, University of Haifa, 31905 Haifa, Israel\\[-0.8ex]
\small\texttt{tmansour@univ.haifa.ac.il}\\[1.8ex]
Mark Shattuck\\
\small Department of Mathematics, University of Tennessee, Knoxville, TN 37996\\[-0.8ex]
\small\texttt{shattuck@math.utk.edu}\\[1.8ex]}

\end{center}

\begin{abstract}
Recently, a new class of words, denoted by $\mathcal{L}_n$, was shown to be in bijection with a subset of the Dyck paths of length $2n$ having cardinality the Catalan number $C_{n-1}$.  Here, we consider statistics on $\mathcal{L}_n$ recording the number of occurrences of a letter $i$.  In the cases $i=0$ and $i=1$, we are able to determine explicit expressions for the number of members of $\mathcal{L}_n$ containing a given number of zeros or ones, which generalizes a prior result.  To do so, we make use of recurrences to derive a functional equation satisfied by the generating function, which we solve by a new method employing Chebyshev polynomials.  Recurrences and generating function formulas are also provided in the case of general $i$.\\
\end{abstract}



\def\P{POGP}
\def\A{\mathcal{A}}
\def\SS{\frak S}
\def\Ps{POGPs}
\def\mn{\mbox{-}}
\def\newop#1{\expandafter\def\csname #1\endcsname{\mathop{\rm #1}\nolimits}}
\newop{MND}
%===========================================================================

%===========================================================================

\section{Introduction}

For a positive integer $n$, let $\mathcal{L}_n$ denote the set of all words $a=(a_1,a_2,\ldots,a_n)$ over the alphabet of non-negative integers satisfying
\begin{itemize}
\item[] (i) $a_{i+1} \geq a_i-1$ for $1 \leq i <n$, and\\
\item[] (ii) if $a_i=k>0$ with $i$ minimal, then there exist $i_1<i<i_2$ such that $a_{i_1}=a_{i_2}=k-1$.
    \end{itemize}
Property (i) states that there are no drops of size greater than one, while (ii) says that the left-most occurrence of each $k>0$ has $k-1$ somewhere to its left and somewhere to its right.  For example, there are $5$ members of $\mathcal{L}_4$, namely,
$$000,0100,0010,0110,0101,$$
and $14$ members of $\mathcal{L}_5$, namely,
\begin{align*}
&00000,01000,00100,00010,01100,01010,01001,\\
&00110,00101,01110,01101,01011,01210,01021.
\end{align*}

The set $\mathcal{L}_n$ was recently introduced by Albert, Ru\v{s}kuc, and Vatter \cite{ARV}, who raised the question \cite{V} of finding a bijection between $\mathcal{L}_n$ and any family of classical Catalan objects.  Stump \cite{St} found such a bijection between $\mathcal{L}_n$ and the Dyck paths of length $2n-2$, proving that this class of words belongs to the Catalan family.  The key step in his bijection was the zeta map, which is an important map in the study of $q,t$-Catalan numbers.  At times, we will refer to sequences satisfying properties (i) and (ii) above as \emph{Catalan words}.

In this paper, we prove refinements of Stump's result, which are obtained by counting the members of $\mathcal{L}_n$ according to the number of occurrences of a given letter $i$.  Of particular interest is the case $i=0$, where it is shown that there are $a(n,m)=\frac{m-1}{n-1}\binom{2n-m-2}{n-2}$ Catalan words of length $n$ containing $m$ zeros. To prove this result, we first determine recurrence relations satisfied by the numbers $a(n,m)$.  In terms of generating functions, our recurrences may be expressed in functional equation form as
$$A(x,v)=\frac{xv}{1-xv}(1-A(x,1))+\frac{xv}{1-xv}A(x,1/(1-xv)).$$
This equation is apparently not solvable by the common techniques for solving functional equations, including the kernel method.  So to determine a solution, we iterate it and then are able to express the resulting infinite series expansion in terms of Chebyshev polynomials, the properties of which allow for a significant simplification.  Using our answer for $A(x,v)$, one can then find a simple expression for the generating function of the sequence that counts Catalan words of length $n$ containing $m$ ones as well as an explicit formula for this sequence.

The paper is divided as follows.  In the next section, we find recurrences and closed-form expressions for the number of members of $\mathcal{L}_n$ containing a fixed number of zeros or ones.  In the third section, we consider the more general problem of counting Catalan words according to the number of occurrences of any letter $i>0$.  Here, recurrences and explicit expressions for the generating functions are provided, the latter of which may be expressed in terms of Chebyshev polynomials.

We recall now the definitions of two sequences.  Given $n \geq 0$, let $C_n=\frac{1}{n+1}\binom{2n}{n}$ denote the $n$-th Catalan number \cite{Stan} and
$$C(x):=\sum_{n \geq 0}C_nx^n=\frac{1-\sqrt{1-4x}}{2x}.$$
See A000108 in \cite{Sl} for further information on these numbers.  Stump \cite[Theorem 1]{St} showed that $|\mathcal{L}_n|=C_{n-1}$ for $n \geq 1$.  The Chebyshev polynomials of the second kind \cite{Ri}, denoted by $U_n(x)$, are defined by the recurrence
$$U_n(x)=2xU_{n-1}(x)-U_{n-2}(x), \qquad n \geq 2,$$
along with the initial values $U_0(x)=1$ and $U_1(x)=2x$.

\section{Counting $\mathcal{L}_n$ by the number of zeros or ones}

\subsection{Recurrences}

By a \emph{descent} in a word $w=w_1w_2\cdots w_m$, we will mean an index $i \in [m-1]$ such that $w_i>w_{i+1}$.  Let $a(n,m,k)$ denote the number of Catalan words of length $n$ having $m$ zeros and $k$ descents.  Note that $a(n,m,k)$ can assume non-zero values only when $1 \leq m \leq n$ and $0\leq k <n$.  The numbers $a(n,m,k)$ satisfy the following recurrence.

\begin{proposition}\label{pr1}
If $n > m \geq 1$ and $k \geq 1$, then
\begin{equation}\label{pre1}
a(n,m,k)=\sum_{d=1}^\ell \sum_{j=1}^{n-m} \binom{m-1}{d}\binom{j}{d}a(n-m,j,k-d),
\end{equation}
where $\ell=\min\{m,k\}$, with $a(n,n,k)=\delta_{k,0}$ if $n >k \geq 0$ and $a(n,m,0)=\delta_{n,m}$ if $n \geq m \geq 1$.
\end{proposition}
\begin{proof}
The boundary conditions are easily verified.  To show \eqref{pre1}, let $\mathcal{L}_{n,m,k}$ denote the subset of $\mathcal{L}_n$ enumerated by $a(n,m,k)$. Given $\lambda \in \mathcal{L}_{n,m,k}$, let $\lambda'$ denote the Catalan word of length $n-m$ obtained by removing the $m$ zeros from $\lambda$, concatenating the remaining letters, and subtracting one from each of these letters.  We will refer to $\lambda'$ as the \emph{reduction of} $\lambda$.  Note that $\lambda' \in \mathcal{L}_{n-m,j,k-d}$ for some $d \in [\ell]$ and $j \in [n-m]$.  We argue that there are $\binom{m-1}{d}\binom{j}{d}a(n-m,j,k-d)$ members of $\mathcal{L}_{n,m,k}$ whose reduction belongs to $\mathcal{L}_{n-m,j,k-d}$, whence \eqref{pre1} follows from summing over $d$ and $j$.

To do so, we will show for each $\lambda' \in \mathcal{L}_{n-m,j,k-d}$ that there are $\binom{m-1}{d}\binom{j}{d}$ members of $\mathcal{L}_{n,m,k}$ whose reduction is $\lambda'$.  We first increment each letter of $\lambda'$ by one and then select $d$ of the $j$ ones in the resulting word, after each of which we write a zero.  Let $\lambda^*$ denote the resulting word of length $n-m+d$.  Note that $\lambda'$ having $k-d$ descents implies $\lambda^*$ has $k$ descents.  We now create members of $\mathcal{L}_{n,m,k}$ from $\lambda^*$ by writing zeros at the beginning or in positions directly following zeros of $\lambda^*$, with at least one zero at the beginning.
(Note that insertion of zeros elsewhere into $\lambda^*$ would create too many descents in the resulting word.)  Write a zero at the beginning of $\lambda^*$ and then insert $m-d-1$ additional zeros in positions of $0\lambda^*$ directly following zeros (with more than one zero in a position being allowed), which can be done in
$$\binom{m-d-1+(d+1)-1}{(d+1)-1}=\binom{m-1}{d}$$
ways.  In this manner, each $\lambda' \in \mathcal{L}_{n-m,j,k-d}$ is seen to give rise to $\binom{m-1}{d}\binom{j}{d}$ members of $\mathcal{L}_{n,m,k}$, which completes the proof of \eqref{pre1}.
\end{proof}

If $1 \leq m \leq n$, then let $a(n,m)$ denote the number of Catalan words of length $n$ that have $m$ zeros. Summing \eqref{pre1} over $k$ and using Vandermonde's identity \cite[p. 169]{GKP} gives the following recurrence for $a(n,m)$.

\begin{corollary}
The numbers $a(n,m)$ satisfy the recurrence
\begin{equation}\label{coe1}
a(n,m)=\sum_{j=1}^{n-m}\left(\binom{j+m-1}{j}-1\right)a(n-m,j), \qquad n \geq 2, \quad 1 \leq m \leq n-1,
\end{equation}
with $a(n,n)=1$ for all $n \geq 1$.
\end{corollary}

Let $b(n,m)$ denote the number of Catalan words of length $n$ containing $m$ ones.  Note that $b(n,m)$ can assume non-zero values only when $0 \leq m <n$.  We have the following relation between the numbers $b(n,m)$ and $a(n,m)$.

\begin{proposition}\label{pr2}
If $n \geq 3$, then
\begin{equation}\label{pr2e1}
b(n,m)=\sum_{i=2}^{n-m}\left(\binom{i+m-1}{m}-1\right)a(n-i,m), \qquad 1 \leq m \leq n-2,
\end{equation}
with $b(n,0)=1$ if $n \geq 1$ and $b(n,n-1)=0$ if $n \geq 2$.
\end{proposition}
\begin{proof}
The boundary conditions follow from the definitions.  To show \eqref{pr2e1}, suppose $1 \leq m \leq n-2$ and let $b(n,m,i)$ denote the number of members of $\mathcal{L}_n$ containing $m$ ones and $i$ zeros.  To complete the proof, it suffices to show
\begin{equation}\label{pr2e2}
b(n,m,i)=\left(\binom{i+m-1}{m}-1\right)a(n-i,m), \qquad 2 \leq i \leq n-m.
\end{equation}

To do so, suppose $\lambda \in \mathcal{L}_{n-i}$ contains $m$ zeros.  Let $\lambda^*$ denote the word obtained from $\lambda$ by increasing each letter by one.  We form members of $\mathcal{L}_n$ enumerated by $b(n,m,i)$ from $\lambda^*$ by writing zeros at the beginning or in positions directly following ones.  This is equivalent to the problem of distributing $i-1$ zeros in $m+1$ positions, with the requirement that not all of the zeros go in the left-most position.  Thus, there are $\binom{i+m-1}{m}-1$ members of $\mathcal{L}_n$ enumerated by $b(n,m,i)$ that arise in this way and grouping together members associated with the same $\lambda$ gives \eqref{pr2e2}.
\end{proof}

\subsection{Counting by number of zeros or ones}

The main result of this section is as follows.

\begin{theorem}\label{th0}
We have
\begin{equation}\label{th0e1}
a(n,m)=\frac{m-1}{n-1}\binom{2n-m-2}{n-2}, \qquad 2 \leq m \leq n.
\end{equation}
\end{theorem}

To show this, we convert the recurrence \eqref{coe1} above to a functional equation, which we solve using Chebyshev polynomials.  Note that formula \eqref{th0e1} follows from combining \eqref{th2e1} below with \cite[Eq. 2.5.16]{W}, which states
$$C(x)^m=\sum_{n \geq 0} \frac{m(2n+m-1)!}{n!(n+m)!}x^n, \qquad m \geq 1.$$

\textbf{Remark 1:} As shown in \cite[Proposition 2]{St}, the statistic recording the number of children of the root on the class of rooted planar trees of size $n+1$ in which the only crucial vertex is the root has the same distribution as the number of zeros on
$\mathcal{L}_n$.  Hence, by \eqref{th0e1}, there are $\frac{m-1}{n-1}\binom{2n-m-2}{n-2}$ such rooted planar trees in which the root has $m$ children.

We now proceed with the proof of Theorem \ref{th0}.  Let $$A(x,v)=\sum_{n\geq1}\sum_{m=1}^n a(n,m)x^nv^m.$$
Then $A(x,v)$ satisfies the following functional equation.

\begin{lemma}\label{l1}
We have
\begin{equation}\label{l1e1}
A(x,v)=\frac{xv}{1-xv}(1-xC(x))+\frac{xv}{1-xv}A(x,1/(1-xv)).
\end{equation}
\end{lemma}
\begin{proof}
Multiplying \eqref{coe1} by $x^nv^m$ and summing over $1\leq m\leq n-1$ gives
\begin{align*}
A(x,v)-\sum_{n\geq1}a(n,n)x^nv^n
=\sum_{i\geq1}\sum_{j=1}^i\left[\sum_{\ell\geq0}\left(\binom{j+\ell}{\ell}-1\right)x^{i+\ell+1}v^{\ell+1}\right]a(i,j),
\end{align*}
which is equivalent to
\begin{align}
A(x,v)&=\frac{xv}{1-xv}+\sum_{i\geq1}\sum_{j=1}^i\left(\frac{x^{i+1}v}{(1-xv)^{j+1}}-\frac{x^{i+1}v}{1-xv}\right)a(i,j)\notag\\
&=\frac{xv}{1-xv}+\frac{xv}{1-xv}(A(x,1/(1-xv))-A(x,1)).\label{l1e2}
\end{align}
From \cite[Theorem 1]{St}, we have $A(x,1)=\sum_{n \geq 1}C_{n-1}x^n=xC(x)$, so that \eqref{l1e1} follows from \eqref{l1e2}.
\end{proof}

We now express $A(x,v)$ as an infinite series involving Chebyshev polynomials.

\begin{lemma}\label{l2}
The generating function $A(x,v)$ is given by
\begin{equation}\label{l2e1}
A(x,v)=\frac{v}{C(x)}\sum_{j\geq1}\frac{\sqrt{x}}{\left(U_{j-1}(t)-v\sqrt{x}U_{j-2}(t)\right)\left(U_{j}(t)-v\sqrt{x}U_{j-1}(t)\right)},
\end{equation}
where $t=\frac{1}{2\sqrt{x}}.$
\end{lemma}
\begin{proof}
Let $L_{-1}(x,v)=v$ and $L_j(x,v)=1/(1-xL_{j-1}(x,v))$ for $j\geq0$. By \eqref{l1e1}, we have
\begin{align}
A(x,L_j(x,v))&=\frac{xL_j(x,v)}{1-xL_j(x,v)}(1-xC(x))+\frac{xL_j(x,v)}{1-xL_j(x,v)}A(x,1/(1-xL_j(x,v)))\notag\\
&=xL_j(x,v)L_{j+1}(x,v)(1-xC(x))+xL_j(x,v)L_{j+1}(x,v)A(x,L_{j+1}(x,v))\notag.
\end{align}
Iterating this last equation an infinite number of times gives
\begin{align}
A(x,L_{-1}(x,v))
&=xL_{-1}(x,v)L_0(x,v)(1-xC(x))+xL_{-1}(x,v)L_0(x,v)A(x,L_0(x,v))\notag\\
&=(1-xC(x))\sum_{j=1}^2x^jL_{-1}(x,v)L_0^2(x,v)\cdots L_{j-2}^2(x,v)L_{j-1}(x,v)\notag\\
&\quad +x^2L_{-1}(x,v)L_0^2(x,v)L_1(x,v)A(x,L_1(x,v))\notag\\
&=\cdots\notag\\
&=(1-xC(x))\sum_{j\geq1}x^jL_{-1}(x,v)L_0^2(x,v)\cdots L_{j-2}^2(x,v)L_{j-1}(x,v).\label{l2e3}
\end{align}
By induction on $j$ and the recurrence for $U_j(x)$, one can show
\begin{equation}\label{l2e4}
L_j(x,v)=\frac{\frac{1-xv}{\sqrt{x}}U_{j-1}(t)-U_{j-2}(t)}{\sqrt{x}\left(\frac{1-xv}{\sqrt{x}}U_{j}(t)-U_{j-1}(t)\right)}, \qquad j \geq -1.
\end{equation}

Substituting \eqref{l2e4} into \eqref{l2e3} and observing some cancellation, we obtain
\begin{equation}\label{l2e5}
A(x,v)=\frac{v}{C(x)}\sum_{j\geq1}\frac{\sqrt{x}}{\left(\frac{1-xv}{\sqrt{x}}U_{j-2}(t)-U_{j-3}(t)\right)\left(\frac{1-xv}{\sqrt{x}}U_{j-1}(t)-U_{j-2}(t)\right)},
\end{equation}
where we have used $1-xC(x)=C(x)^{-1}$.  Formula \eqref{l2e1} now follows from \eqref{l2e5} and the defining relation $U_m(t)=\frac{1}{\sqrt{x}}U_m(t)-U_{m-1}(t)$.
\end{proof}

As a consequence of Lemma \ref{l2}, we obtain the following identity.

\begin{corollary}\label{co1}
We have
\begin{equation}\label{co1e1}
\sum_{j\geq1}\frac{1}{\sqrt{x}U_{j}(t)U_{j+1}(t)}=xC^2(x).
\end{equation}
\end{corollary}
\begin{proof}
Taking $v=1$ in \eqref{l2e1} gives
\begin{align*}
A(x,1)&=\frac{1}{C(x)}\sum_{j\geq1}\frac{\sqrt{x}}{(U_{j-1}(t)-\sqrt{x}U_{j-2}(t))(U_{j}(t)-\sqrt{x}U_{j-1}(t))}\\
&=\frac{1}{C(x)}\sum_{j\geq1}\frac{1}{\sqrt{x}U_{j}(t)U_{j+1}(t)},
\end{align*}
which implies \eqref{co1e1} since $A(x,1)=xC(x)$.
\end{proof}

We can now determine a closed form expression for $A(x,v)$.

\begin{theorem}\label{th2}
We have
\begin{equation}\label{th2e1}
A_m(x):=\sum_{n \geq m}a(n,m)x^n=x^mC^{m-1}(x), \qquad m \geq 1,
\end{equation}
and hence
\begin{equation}\label{th2e2}
A(x,v)=\sum_{m\geq 1}A_m(x)v^m=\frac{xv}{1-xvC(x)}.
\end{equation}
\end{theorem}
\begin{proof}
Formula \eqref{th2e2} follows immediately from \eqref{th2e1}, so we only need to show \eqref{th2e1}.  Expanding \eqref{l2e1} using a geometric series yields
\begin{align*}
A(x,v)
&=\frac{1}{C(x)}\sum_{j\geq1}\sum_{k,\ell\geq0}\frac{v^{k+\ell+1}\sqrt{x}^{k+\ell+1}U_{j-2}^k(t)U_{j-1}^{\ell}(t)}{U_{j-1}^{k+1}(t)U_j^{\ell+1}(t)},
\end{align*}
and extracting the coefficient of $v^m$ in the last expression gives
\begin{align*}
A_m(x)
&=\frac{1}{C(x)}\sum_{j\geq1}\sum_{k=0}^{m-1}\frac{\sqrt{x}^mU_{j-2}^k(t)U_{j-1}^{m-2-2k}(t)}{U_j^{m-k}(t)}\\
&=\frac{\sqrt{x}^m}{C(x)}\sum_{j\geq1}\frac{U_{j-2}^m(t)U_j^m(t)-U_{j-1}^{2m}(t)}{U_{j-1}^{m}(t)U_j^m(t)(U_{j-2}(t)U_j(t)-U_{j-1}^2(t))}.
\end{align*}
By the fact that $U_{j-2}(t)U_j(t)-U_{j-1}^2(t))=-1$, we obtain
\begin{align*}
A_m(x)
&=\frac{\sqrt{x}^m}{C(x)}\sum_{j\geq1}\frac{U_{j-1}^{2m}(t)-U_{j-2}^m(t)U_j^m(t)}{U_{j-1}^{m}(t)U_j^m(t)}
=\frac{x^m}{C(x)}\sum_{j\geq1}\left(\frac{U_{j-1}^{m}(t)}{\sqrt{x}^mU_j^m(t)}-\frac{U_{j-2}^m(t)}{\sqrt{x}^mU_{j-1}^{m}(t)}\right)\\
&=\frac{x^m}{C(x)}\lim_{j\rightarrow\infty}\left(\frac{U_{j-1}^{m}(t)}{\sqrt{x}^mU_j^m(t)}\right).
\end{align*}
Formula \eqref{th2e1} now follows from the fact that $\lim_{j\rightarrow\infty}\left(\frac{U_{j-1}(t)}{\sqrt{x}U_j(t)}\right)=C(x)$ (see \cite{MV1}).
\end{proof}

Combining Lemma \ref{l2} and Theorem \ref{th2}, we obtain the following identity.

\begin{corollary}\label{co2}
We have
\begin{equation}\label{co2e1}
\sum_{j\geq1}\frac{1}{\sqrt{x}\left(U_{j-1}(t)-v\sqrt{x}U_{j-2}(t)\right)\left(U_{j}(t)-v\sqrt{x}U_{j-1}(t)\right)}=\sum_{m\geq1}x^{m-1}v^{m-1}C^{m}(x)
=\frac{C(x)}{1-xvC(x)}.
\end{equation}
\end{corollary}

Taking $v=1$ in \eqref{th2e1} and summing over odd $m$, we obtain the following result.

\begin{corollary}\label{co3}
We have
\begin{equation}\label{co3e1}
\sum_{m \geq 0}A_{2m+1}(x)=\frac{x}{1-x^2C^2(x)}=\frac{1-\sqrt{1-4x}}{3-\sqrt{1-4x}},
\end{equation}
and thus the number of members of $\mathcal{L}_n$ having an odd number of zeros is given by A000957 in \cite{Sl} (the Fine Numbers).
\end{corollary}

One can also give an explicit generating function formula for the statistic recording the number of ones.  Let $$B(x,v)=\sum_{n \geq 1}\sum_{m=0}^{n-1} b(n,m)x^nv^m.$$

\begin{corollary}\label{co4}
We have
\begin{equation}\label{co4e1}
B(x,v)=\frac{x(1-x+x^2v+xv(x-2)C(x)+x^2v^2C^2(x))}{(1-x)(1-xvC(x))(1-x-xvC(x))}
\end{equation}
and
\begin{equation}\label{co4e1.5}
b(n,m)=\sum_{j=m}^{n-1} \frac{m-1}{n+m-j-2}\binom{2n+m-2j-4}{n-j-1}\left(\binom{j}{m}-1\right), \qquad 2 \leq m \leq n.
\end{equation}
\end{corollary}
\begin{proof}
By \eqref{pr2e1} and the definition of $A(x,v)$, we have
\begin{align*}
B(x,v)&=\frac{x}{1-x}+\sum_{n \geq 3}\sum_{m=1}^{n-2}b(n,m)x^nv^m\\
&=\frac{x}{1-x}+\sum_{n \geq 3}\sum_{m=1}^{n-2} \sum_{i=2}^{n-m}\left(\binom{i+m-1}{m}-1\right)a(n-i,m)x^nv^m\\
&=\frac{x}{1-x} + \frac{x}{1-x}\left(A(x,v/(1-x))-A(x,v)\right).
\end{align*}
Using \eqref{th2e2} in the last expression, and simplifying, gives \eqref{co4e1}.  Formula \eqref{co4e1.5} follows either from combining \eqref{pr2e1} and \eqref{th0e1} or from extracting the coefficient of $x^nv^m$ in \eqref{co4e1}.
\end{proof}

\textbf{Remark 2:} Though it appears that we have used the $v=1$ case of formula \eqref{th2e2} in its proof, this can in fact be avoided by letting $A(x,1)=g(x)$ at the step where we made the substitution $A(x,1)=xC(x)$ in \eqref{l1e2}.  Doing so, then one gets, instead of \eqref{th2e1}, the relation
$$A_m(x)=x^mC^m(x)(1-g(x)), \qquad m \geq 1.$$
Summing over $m \geq 1$ gives
$$\frac{g(x)}{1-g(x)}=\frac{xC(x)}{1-xC(x)},$$
and thus $A(x,1)=xC(x)$.  Therefore, our proof demonstrates the fact $|\mathcal{L}_n|=C_{n-1}$, which was shown bijectively in \cite{St}.

\textbf{Remark 3:} We were unable to find an explicit formula for the numbers $a(n,m,k)$ satisfying \eqref{pre1}, which would provide a refinement of Theorem \ref{th0}.

\section{Counting $\mathcal{L}_n$ by the number of occurrences of any letter $i>0$}

We first define an array $a_i(n,s,t)$, where $s>0$.  Given $n \geq 2$, $1 \leq i \leq \lfloor\frac{n-1}{2}\rfloor$, $2 \leq t \leq n$, and $1 \leq s \leq n-t-2i+2$, let $a_{i}(n,s,t)$ denote the number of members of $\mathcal{L}_n$ in which there are exactly $s$ occurrences of the letter $i$ and $t$ zeros.  Note that in order for the letter $i$ to occur at all, each of the letters in $[i-1]$ must occur at least twice. Since zero occurs $t$ times, we get the upper bound $s \leq n-t-2(i-1)$.

If $s=0$, then $a_i(n,0,t)$ is the number of members of $\mathcal{L}_n$ not containing any letters $i$ and containing exactly $t$ zeros, where $1\leq t \leq n$ and $i \geq 1$. In this case, there is no restriction on $i$, given $1 \leq t \leq n$.  Furthermore, observe that if $i> \lfloor\frac{n-1}{2}\rfloor$, then $a_i(n,0,t)=a(n,t)$, where $a(n,t)$ is as in the previous section.  By contrast, we have $a_i(n,s,t)=0$ if $s>0$ and $i> \lfloor\frac{n-1}{2}\rfloor$.

Note that the maximum letter in a sequence enumerated by $a_i(n,0,t)$ is at most $i-1$, which implies that the difference $a_{i+1}(n,0,t)-a_i(n,0,t)$ gives the number of members of $\mathcal{L}_n$ whose largest letter is $i$.

The following proposition provides recurrences for the array $a_i(n,s,t)$.

\begin{proposition}\label{pr3}
If $s>0$, then
\begin{equation}\label{rec1}
a_i(n,s,t)=\sum_{\ell=2}^m\left(\binom{\ell+t-1}{\ell}-1\right) a_{i-1}(n-t,s,\ell), \qquad i>1~\text{ and }~ s+t < n-2i+3,
\end{equation}
where $m=n-s-t-2i+4$.  If $i=1$, then
\begin{equation}\label{rec1.5}
a_1(n,s,t)=\left(\binom{s+t-1}{s}-1\right)a(n-t,s), \qquad 2 \leq t \leq n-s.
\end{equation}
If $s=0$, then
\begin{equation}\label{rec2}
a_i(n,0,t)=\delta_{n,t}+\sum_{\ell=1}^{n-t}\left(\binom{\ell+t-1}{\ell}-1\right) a_{i-1}(n-t,0,\ell), \qquad i\geq 1~\text{ and }~ 1 \leq t \leq n,
\end{equation}
where $a_0(n-t,0,\ell)$ is understood to always be zero.
\end{proposition}
\begin{proof}
First assume $s>0$.  Note that for $i=1$, relation \eqref{rec1.5} is equivalent to \eqref{pr2e2} above.  If $i>1$, suppose $\lambda \in \mathcal{L}_n$ is enumerated by $a_i(n,s,t)$.  Then the reduction of $\lambda$ (see proof of Proposition \ref{pr1} above for definition) is a member $\lambda'$ of $\mathcal{L}_{n-t}$ counted by $a_{i-1}(n-t,s,\ell)$ for some $2 \leq \ell \leq m$.  On the other hand,
starting with $\lambda'$ as given, there are $\binom{\ell+t-1}{\ell}-1$ members $\lambda \in \mathcal{L}_n$ whose reduction is $\lambda'$, by prior reasoning.  Summing over all possible $\ell$ then gives \eqref{rec1}.  Similar reasoning applies to \eqref{rec2} in the case when $i>1$ and $t<n$.  If $i=1$ or $t=n$, then it is obvious that $a_i(n,0,t)=\delta_{n,t}$, which completes the proof.
\end{proof}

We now seek a generating function formula for the arrays $a_i(n,s,t)$.  First suppose $s>0$.  Define $A_i(x|s,t)=\sum_{n\geq s+t+2i-2}a_i(n,s,t)x^n$ for $i \geq 1$, with $A(x|s)=\sum_{n\geq s}a(n,s)x^n$. Multiplying \eqref{rec1} and \eqref{rec1.5} by $x^n$ and summing over $n\geq s+t+2i-2$ gives
$$A_i(x|s,t)=x^t\sum_{\ell\geq2}\left[\binom{\ell+t-1}{\ell}-1\right]A_{i-1}(x|s,\ell)$$
and
$$A_1(x|s,t)=x^t\left[\binom{s+t-1}{s}-1\right]A(x|s).$$

Define $A_i(x,w|s)=\sum_{t\geq1}A_i(x|s,t)w^t$. Then the preceding recurrences yield
$$A_i(x,w|s)=\frac{xw}{1-xw}(A_{i-1}(x,1/(1-xw)|s)-A_{i-1}(x,1|s))$$
and
$$A_1(x,w|s)=\frac{xw}{1-xw}\left(\frac{1}{(1-xw)^{s}}-1\right)A(x|s).$$

Define $A_i(x,w,v)=\sum_{s\geq1}A_i(x,w|s)v^s$. Then we have
$$A_i(x,w,v)=\frac{xw}{1-xw}(A_{i-1}(x,1/(1-xw),v)-A_{i-1}(x,1,v)), \qquad i>1,$$
with
$$A_1(x,w,v)=\frac{xw}{1-xw}(A(x,v/(1-xw))-A(x,v)),$$
where $A(x,v)$ is as before.

Define $A(x,w,v,q)=\sum_{i\geq1}A_i(x,w,v)q^i$. Then
$$A(x,w,v,q)=\frac{xwq}{1-xw}(A(x,v/(1-xw))-A(x,v))+\frac{xwq}{1-xw}(A(x,1/(1-xw),v,q)-A(x,1,v,q)).$$
By iterating this last equation an infinite number of times, we obtain
\begin{align*}
A(x,w,v,q)=\sum_{j\geq0}&\frac{xqL_{j-1}(x,w)}{1-xL_{j-1}(x,w)}(A(x,v/(1-xL_{j-1}(x,w))-A(x,v)\\
&-A(x,1,v,q))\prod_{i=0}^{j-1}\frac{xqL_{i-1}(x,w)}{1-xL_{i-1}(x,w)},
\end{align*}
where $L_j(x,w)$ is as in the previous section.  Since $1-xL_j(x,w)=\frac{1}{L_{j+1}(x,w)}$, the last equation may be rewritten as
\begin{align*}
&A(x,w,v,q)\\
&=\sum_{j\geq0}(A(x,vL_{j}(x,w))-A(x,v)-A(x,1,v,q))\prod_{i=0}^{j}xqL_{i-1}(x,w)L_{i}(x,w)\\
&=\sum_{j\geq0}x^{j+1}q^{j+1}(A(x,vL_{j}(x,w))-A(x,v)-A(x,1,v,q))L_{-1}(x,w)L_0^2(x,w)L_1^2(x,w)\cdots\\ &\qquad \quad L_{j-1}^2(x,w)L_j(x,w)\\
&=\sum_{j\geq0}\frac{w\sqrt{x}q^{j+1}(A(x,vL_{j}(x,w))-A(x,v)-A(x,1,v,q))}{(U_{j+1}(t)-w\sqrt{x}U_j(t))(U_{j}(t)-w\sqrt{x}U_{j-1}(t))},
\end{align*}
where $t=\frac{1}{2\sqrt{x}}$.  Setting $w=1$ gives
\begin{align*}
A(x,1,v,q)&=\sum_{j\geq0}\frac{q^{j+1}(A(x,vL_{j}(x,1))-A(x,v)-A(x,1,v,q))}{\sqrt{x}U_{j+2}(t)U_{j+1}(t)},
\end{align*}
and solving for $A(x,1,v,q)$ implies
\begin{align*}
A(x,1,v,q)&=\frac{\sum_{j\geq0}\frac{q^{j+1}(A(x,vL_{j}(x,1))-A(x,v))}{\sqrt{x}U_{j+2}(t)U_{j+1}(t)}}{1+\sum_{j\geq0}\frac{q^{j+1}}{\sqrt{x}U_{j+2}(t)U_{j+1}(t)}}.
\end{align*}
Hence, we can state the following result.

\begin{theorem}\label{th3}
The generating function $A(x,w,v,q)$ is given by
\begin{equation}\label{th3e1}
A(x,w,v,q)=q\sum_{j\geq0}\frac{w\sqrt{x}q^{j}\left(A(x,vL_{j}(x,w))-A(x,v)-\frac{q\sum_{i\geq0}\frac{q^{i}(A(x,vL_{i}(x,1))-A(x,v))}{\sqrt{x}U_{i+2}(t)U_{i+1}(t)}}{1+q\sum_{i\geq0}\frac{q^{i}}{\sqrt{x}U_{i+2}(t)U_{i+1}(t)}}\right)}{(U_{j+1}(t)-w\sqrt{x}U_j(t))(U_{j}(t)-w\sqrt{x}U_{j-1}(t))},
\end{equation}
where $t=\frac{1}{2\sqrt{x}}$, $A(x,v)=\frac{xv}{1-xvC(x)}$ and $L_j(x,w)=1/(1-xL_{j-1}(x,w))$ for $j\geq1$, with $L_{0}(x,w)=\frac{1}{1-xw}$.
\end{theorem}

Note that the coefficient of $q^i$ in $A(x,1,v,q)$ is the generating function counting members of $\mathcal{L}_n$ containing at least one $i$ according to the number of occurrences of $i$, where $i>0$.

We now consider the case when $s=0$.  Let $A_i(x|0,t)=\sum_{n\geq t} a_i(n,0,t)x^n$. Then recurrence \eqref{rec2} can be written as
$$A_i(x|0,t)=x^t+x^t\sum_{\ell\geq1}\left[\binom{\ell+t-1}{\ell}-1\right]A_{i-1}(x|0,\ell), \qquad i \geq 1,$$
with $A_0(x|0,t)=0$.
Let $A_i(x,w|0)=\sum_{t\geq1}A_i(x|0,t)w^t$.  Multiplying the last recurrence by $w^t$ and summing over $t\geq1$, we obtain
$$A_i(x,w|0)=\frac{xw}{1-xw}\left(1+A_{i-1}(x,1/(1-xw)|0)-A_{i-1}(x,1|0)\right), \qquad i \geq 1,$$
with $A_0(x,w|0)=0$.
Define $A(x,w,q|0)=\sum_{i\geq0}A_i(x,w|0)q^i$. Thus, we have

$$A(x,w,q|0)=\frac{xwq}{(1-xw)(1-q)}+\frac{xwq}{1-xw}\left(A(x,1/(1-xw),q|0)-A(x,1,q|0)\right).$$

By iterating the preceding equation an infinite number of times, we obtain
$$A(x,w,q|0)=\sum_{j\geq0}\prod_{i=1}^{j+1}\frac{xqL_{i-2}(x,w)}{1-xL_{i-2}(x,w)}\left(\frac{1}{1-q}-A(x,1,q|0)\right),$$
which is equivalent to
$$A(x,w,q|0)=\sum_{j\geq0}x^{j+1}q^{j+1}L_{-1}(x,w)L_0^2(x,q)L_1^2(x,q)\cdots L_{j-1}^2(x,w)L_j(x,w)\left(\frac{1}{1-q}-A(x,1,q|0)\right),$$
since $1-xL_j(x,w)=\frac{1}{L_{j+1}(x,w)}$.
By \eqref{l2e4}, we have
$$A(x,w,q|0)=\sum_{j\geq0}\frac{w\sqrt{x}q^{j+1}}{(U_{j+1}(t)-w\sqrt{x}U_j(t))(U_{j}(t)-w\sqrt{x}U_{j-1}(t))}\left(\frac{1}{1-q}-A(x,1,q|0)\right).$$

Setting $w=1$ gives
\begin{align*}
A(x,1,q|0)
&=\sum_{j\geq0}\frac{\sqrt{x}q^{j+1}}{(U_{j+1}(t)-\sqrt{x}U_j(t))(U_{j}(t)-\sqrt{x}U_{j-1}(t))}\left(\frac{1}{1-q}-A(x,1,q|0)\right)\\
&=\sum_{j\geq0}\frac{q^{j+1}}{\sqrt{x}U_{j+2}(t)U_{j+1}(t)}\left(\frac{1}{1-q}-A(x,1,q|0)\right),
\end{align*}
and solving for $A(x,1,q|0)$ implies
\begin{align*}
A(x,1,q|0)
&=\frac{\frac{q}{1-q}\sum_{j\geq0}\frac{q^{j}}{\sqrt{x}U_{j+2}(t)U_{j+1}(t)}}{1+q\sum_{j\geq0}\frac{q^{j}}{\sqrt{x}U_{j+2}(t)U_{j+1}(t)}}.
\end{align*}
Hence, we can state the following result.

\begin{theorem}\label{th4}
The generating function $A(x,w,q|0)$ is given by
\begin{equation}\label{th4e1}
A(x,w,q|0)=\frac{q\sum_{j\geq0}\frac{w\sqrt{x}q^{j}}{(U_{j+1}(t)-w\sqrt{x}U_j(t))(U_{j}(t)-w\sqrt{x}U_{j-1}(t))}}{(1-q)\left(1+q\sum_{j\geq0}\frac{q^{j}}{\sqrt{x}U_{j+2}(t)U_{j+1}(t)}\right)},
\end{equation}
where $t=\frac{1}{2\sqrt{x}}$.
\end{theorem}

\end{document}